\documentclass[11pt,a4paper]{article}

\usepackage[left=4cm, top=3.5cm,bottom=4cm,right=4cm]{geometry}

\usepackage{graphicx}
\usepackage{latexsym}
\usepackage{amsbsy}
\usepackage{amsthm}
\usepackage{amssymb}
\usepackage{amsfonts}
\usepackage{dsfont}
\usepackage{amsmath}

\newtheorem{theorem}{Theorem}[section]

\begin{document}

\title{\bfseries The Duality of the Volumes and the Numbers of Vertices of Random Polytopes}

\author{Christian Buchta}

\date{}

\maketitle

\begin{abstract}
\noindent
An identity due to Efron dating from 1965 relates the expected volume of the convex hull of $n$ random points to the expected number of vertices of the convex hull of $n+1$ random points. Forty years later this identity was extended from expected values to higher moments. The generalized identity has attracted considerable interest. Whereas the left-hand side of the generalized identity---concerning the volume---has an immediate geometric interpretation, this is not the case for the right-hand side---concerning the number of vertices. A transformation of the right-hand side applying an identity for elementary symmetric polynomials overcomes the blemish. The arising formula reveals a duality between the volumes and the numbers of vertices of random polytopes.
\end{abstract}

\section{Introduction}
\label{intro}
Write $\mathcal{K}^{d}$ for the set of all convex bodies (convex compact sets with non-empty interiors) in $\mathbb{R}^{d}$. Fix $K\in \mathcal{K}^{d}$, and choose points  $x_{1},\dots , x_{n}\in K$ randomly, independently, and according to the uniform distribution on $K$. The set $K_{n}= \textrm{conv}\left\lbrace x_{1},\dots , x_{n}\right\rbrace $ is a random polytope. Denote the volume of $K_{n}$ by the random variable $V_{n}$ and the number 
of vertices of $K_{n}$ by the random variable $N_{n}$.
\vskip 0.5truecm
Efron \cite{RefE}, p.\ 335, Formulae (3.5) and (3.7), proved the identity 
\begin{equation} \label{eq:1}
\frac{EV_{n}}{\textrm{vol}\;K}=1-\frac{E N_{n+1}}{n+1}.
\end{equation}
The left-hand side has an immediate geometric interpretation: Assume that $n$ points have been chosen at random. Then $EV_{n}/\textrm{vol}\;K$ is the probability that a further point chosen at random falls into the convex hull of the $n$ points, which means that it is not a vertex of the convex hull of the $n+1$ points.

Likewise, the right-hand side can be interpreted geometrically: The ratio $EN_{n+1}/(n+1)$ is the probability that any of $n+1$ points chosen at random is a vertex of their convex hull and hence not contained in the convex hull of the remaining $n$ points.

In \cite{RefBu}, p.\ 127, Formula (1.8), Efron's identity  (\ref{eq:1}) was extended to higher moments:
\begin{equation} \label{eq:2}
\frac{EV^{k}_{n}}{(\textrm{vol}\;K)^{k}} = E \prod_{i=1}^{k} \left( 1-\frac{N_{n+k}}{n+i}\right).
\end{equation}

This result has attracted considerable interest; see in particular Subsection 8.2.3 in the book by Schneider and Weil \cite{RefSW}, the survey articles (chapters of books) by Reitzner \cite{RefR3}, Hug \cite{RefH}, and Schneider \cite{RefS2}, as well as the research articles by Cowan \cite{RefCo}, Groeneboom \cite{RefG}, Beermann and Reitzner \cite{RefBeR}, and Kabluchko, Last, and Zaporozhets \cite{RefKLZ}. The identity (\ref{eq:2}) is also mentioned in a survey paper (book chapter) by Calka \cite{RefCa} and in research papers by Reitzner \cite{RefR1}, Kabluchko, Marynych, Temesvari, and Th\"{a}le \cite{RefKMTT}, as well as Brunel \cite{RefBr}. A consequence of the identity is stated in a survey article (book chapter) by B\'ar\'any \cite{RefBa1}, in a further survey article by B\'ar\'any \cite{RefBa2}, and in a survey article---different from the one mentioned above---by Schneider \cite{RefS1}, as well as in research articles by Reitzner \cite{RefR2} and B\'ar\'any and Reitzner \cite{RefBaR1,RefBaR2}.

The extension of the geometric interpretation of the left-hand side of (\ref{eq:1}) to the left-hand side of (\ref{eq:2}) is straightforward: Assume again that $n$ points have been chosen at random. Then  $EV_{n}^{k}/(\textrm{vol}\;K)^{k}$ is the probability that each of $k$ further points chosen at random falls into the convex hull of the $n$ points, which means that none of the specified $k$ points among the $n+k$ points chosen at random is a vertex of the convex hull of the $n+k$ points.

The right-hand side of (\ref{eq:2}) has no obvious geometric meaning which explains that the right-hand side is equivalent to the left-hand side. Writing it in the form
\begin{equation*}
\frac{(-1)^{k}n!}{(n+k)!} E \prod_{i=1}^{k} \left( N_{n+k}-(n+i)\right)
\end{equation*}
we identify $ \prod_{i=1}^{k} \left( N_{n+k}-(n+i)\right)$ as the generating function of the elementary symmetric polynomials 
$\sigma_{j}(n+1,\ldots,n+k),\;j=0,\ldots,k$, in the variables $n+1,\ldots,n+k$. It will turn out in the proof of Theorem \ref{theorem1} that decomposing these polynomials in a suitable way gives rise to a transformation of (\ref{eq:2}) into
\begin{equation} \label{eq:3}
\frac{EV_{n}^{k}}{(\textrm{vol}\;K)^{k}}=
\sum\limits_{j=0}^{k} (-1)^{j} {k\choose j} \frac{E(N_{n+k})_{(j)}}{(n+k)_{(j)}},
\end{equation}
where $E(N_{n+k})_{(j)}=EN_{n+k}(N_{n+k}-1)\cdots(N_{n+k}-j+1)$ is the $j$-th factorial moment of $N_{n+k}$ and $(n+k)_{(j)}=(n+k)(n+k-1)\cdots(n+k-j+1)=\frac{(n+k)!}{(n+k-j)!}$.

The ratio $E(N_{n+k})_{(j)}/(n+k)_{(j)}$ has a simple geometric interpretation: It is just the probability that $j$ specified points in a set of $n+k$ random points are vertices of the convex hull of the $n+k$ points. Consequently, by the inclusion--exclusion principle,
\begin{equation*}
\sum\limits_{j=1}^{k} (-1)^{j-1} {k\choose j} \frac{E(N_{n+k})_{(j)}}{(n+k)_{(j)}}
\end{equation*}
is the probability that at least one of $k$ specified points among $n+k$ points chosen at random is a vertex of the convex hull of the $n+k$ points. The complementary probability, i.e.\ the probability that none of the $k$ specified points is a vertex of the convex hull of the $n+k$ points, is just given by the right-hand side of (\ref{eq:3}).

Trivially, (\ref{eq:1}) is equivalent to 
\begin{equation*} 
\frac{EN_{n+1}}{n+1}=1-\frac{EV_{n}}{\textrm{vol}\;K}.
\end{equation*}
We prove that, more generally, the identity
\begin{equation} \label{eq:4}
\frac{E(N_{n+k})_{(k)}}{(n+k)_{(k)}}=\sum\limits_{j=0}^{k} (-1)^{j} {k\choose j} \frac{EV^{j}_{n+k-j}}{(\textrm{vol}\;K)^{j}}
\end{equation}
is dual to the identity (\ref{eq:3}) in the sense that (\ref{eq:3}) implies (\ref{eq:4}), and (\ref{eq:4}) implies (\ref{eq:3}).

\section{Transformation Based on the Decomposition of Certain Elementary Symmetric Polynomials}
\label{sec:1}
Write $\sigma_{j}(x_{1},\dots , x_{k})$ for the $j$-th elementary symmetric polynomial in the variables $x_{1},\dots , x_{k}$, i.e.
\begin{equation*} 
\sigma_{j}(x_{1},\dots , x_{k})=\sum_{1\leq i_{1}<\ldots <i_{j}\leq k}x_{i_{1}}\cdots x_{i_{j}}.
\end{equation*}
The elementary symmetric polynomials generalize the binomial coefficients: $\sigma_{j}(x_{1},\dots , x_{k})= {k\choose j}$ if $x_{1}=\ldots=x_{k}=1$. Correspondingly, the generating function of the binomial coefficients
\begin{equation*} 
(t-1)^{k}=\sum\limits_{j=0}^{k} (-1)^{j} {k\choose j}t^{k-j}
\end{equation*}
extends to 
\begin{equation*}
\prod_{i=1}^{k}(t-x_{i})=\sum\limits_{j=0}^{k}(-1)^{j}\sigma_{j}(x_{1},\ldots,x_{k})t^{k-j},
\end{equation*}
and the recurrence relation ${k\choose j}={k-1\choose j-1}+{k-1\choose j}$ extends to 
\begin{equation} \label{eq:5}
\sigma_{j}(x_{1},\dots , x_{k})=x_{k}\sigma_{j-1}(x_{1},\dots , x_{k-1})+\sigma_{j}(x_{1},\dots , x_{k-1}).
\end{equation}
The generating function and the recurrence relation for the binomial coefficients suggest to define ${k\choose j}=1$ if $j=0$ and ${k\choose j}=0$ if $j\neq0$ and $k<j$. Likewise, the generating function and the recurrence relation for the elementary symmetric polynomials suggest to define $\sigma_{j}(x_{1},\dots , x_{k})=1$ if $j=0$ and $\sigma_{j}(x_{1},\dots , x_{k})=0$ if $j\neq0$ and $k<j$. Here $j$ and $k$ are also allowed to be negative.

To transform the product on the right-hand side of (\ref{eq:2}), each of the elementary symmetric polynomials $\sigma_{j}(n+1,\dots, n+k),\;j=0,\ldots,k$, in the variables $n+1,\ldots, n+k$, which occur there implicitly, is now decomposed into a sum of products of elementary symmetric polynomials such that the first factor is an elementary symmetric polynomial in just the integers $1,2,\ldots,$ whereas the second factor is the elementary symmetric polynomial of maximal degree, i.e.\ just the product of all variables.\vskip0.25cm\noindent
\textbf{Proposition.} \textit{Denote by $\sigma_{j}(x_{1},\dots , x_{k})$ the $j$-th elementary symmetric polynomial in the $k$ variables $x_{1},\dots , x_{k}$. Then}
\begin{equation*} 
\sigma_{j}(x+1,\dots , x+k)= \sum\limits_{i=0}^{j}{k\choose i} \sigma_{j-i}(1,\dots , k-i-1)\sigma_{i}(x+1,\dots , x+i).
\end{equation*}

\begin{proof} 
The formula is correct if $k=1$. Assume that it has been verified for elementary symmetric polynomials in less than $k$ variables. Applying the recurrence relation (\ref{eq:5}) we obtain
\begin{align*}
\sigma_{j}&(x+1,\dots , x+k)\\
= &(x+k)\sum\limits_{i=0}^{j-1}{k-1\choose i}\sigma_{j-i-1}(1,\dots , k-i-2)\sigma_{i}(x+1,\dots , x+i)\\
&+\sum\limits_{i=0}^{j}{k-1\choose i} \sigma_{j-i}(1,\dots , k-i-2)\sigma_{i}(x+1,\dots , x+i).
       \end{align*}
Splitting the factor $x+k$ by which the first sum is multiplied into $x+i+1$ and $k-i-1$, observing that $(x+i+1)\sigma_{i}(x+1,\dots , x+i)=\sigma_{i+1}(x+1,\dots , x+i+1)$, and adapting the summation index, we get
\begin{align*}
\sigma_{j}&(x+1,\dots , x+k)\\
= &\sum\limits_{i=1}^{j}{k-1\choose i-1} \sigma_{j-i}(1,\dots , k-i-1)\sigma_{i}(x+1,\dots , x+i)\\
&+\sum\limits_{i=0}^{j-1}{k-1\choose i}(k-i-1)\sigma_{j-i-1}(1,\dots , k-i-2)\sigma_{i}(x+1,\dots , x+i)\\
&+\sum\limits_{i=0}^{j}{k-1\choose i}\sigma_{j-i}(1,\dots , k-i-2)\sigma_{i}(x+1,\dots , x+i).
 \end{align*}
If the index is taken from $0$ to $j$ in all three sums, only terms are added which have the value zero. According to (\ref{eq:5}),
\begin{equation*} 
(k-i-1)\sigma_{j-i-1}(1,\dots , k-i-2)+\sigma_{j-i}(1,\dots , k-i-2)=\sigma_{j-i}(1,\dots , k-i-1),
\end{equation*}
and as 
${k-1\choose i-1}+{k-1\choose i}={k\choose i}$, the claimed formula arises.
\end{proof}
\begin{theorem}
\label{theorem1}
Let $K\in \mathcal{K}^{d},\,n\in\mathbb{N}$, and $k\in\mathbb{N}$. Then
\begin{equation*} 
\frac{EV_{n}^{k}}{(\textnormal{vol}\;K)^{k}}=\sum\limits_{j=0}^{k}(-1)^{j}{k\choose j}\frac{E(N_{n+k})_{(j)}}{(n+k)_{(j)}},
\end{equation*}
where $E(N_{n+k})_{(j)}=EN_{n+k}(N_{n+k}-1)\cdots(N_{n+k}-j+1)$ is the $j$-th factorial moment of $N_{n+k}$ and $(n+k)_{(j)}=(n+k)(n+k-1)\cdots(n+k-j+1)$.
\end{theorem}
\begin{proof} 
We start from Theorem 1 in \cite{RefBu}, which states that
\begin{equation*}
\frac{EV_{n}^{k}}{(\textrm{vol}\;K)^{k}}= E\prod_{i=1}^{k}\left( 1-\frac{N_{n+k}}{n+i}\right).
\end{equation*}
The right-hand side can equivalently be written in the form
\begin{equation*} 
\frac{n!}{(n+k)!}\sum\limits_{j=0}^{k}(-1)^{k-j}\sigma_{j}(n+1,\dots , n+k)EN^{k-j}_{n+k}.
\end{equation*}
The Proposition transforms this into
\begin{equation*}
\frac{n!}{(n+k)!}\sum\limits_{j=0}^{k}(-1)^{k-j}\sum\limits_{i=0}^{j}{k\choose i}\sigma_{j-i}(1,\dots ,k-i-1)\sigma_{i}(n+1,\dots ,n+i)EN^{k-j}_{n+k}.
       \end{equation*}
Interchanging the order of summation yields
\begin{equation*}
\frac{n!}{(n+k)!}\sum\limits_{i=0}^{k}{k\choose i} \sigma_{i}(n+1,\dots , n+i)\sum\limits_{j=i}^{k}(-1)^{k-j}\sigma_{j-i}(1,\dots,k-i-1)EN^{k-j}_{n+k}.
       \end{equation*}
Since
\begin{equation*} 
\sum\limits_{j=i}^{k}(-1)^{k-j}\sigma_{j-i}(1,\dots ,k-i-1)EN^{k-j}_{n+k}=(-1)^{k-i}E(N_{n+k})_{(k-i)}
\end{equation*}
and
\begin{equation*} 
\frac{n!}{(n+k)!}\sigma_{i}(n+1,\dots , n+i)=\frac{1}{(n+k)_{(k-i)}},
\end{equation*}
we obtain
\begin{equation*} 
\sum\limits_{i=0}^{k}(-1)^{k-i}{k\choose i}\frac{E(N_{n+k})_{(k-i)}}{(n+k)_{(k-i)}},
\end{equation*}
which is identical with the claimed expression.
\end{proof}
The subsequent theorem provides a geometric interpretation of the $j$-th factorial moment of $N_{n+k}\;(j=0,\ldots,k)$ occurring in Theorem \ref{theorem1}. We write $m$ instead of $n+k-j$ hoping to facilitate understanding.

\begin{theorem}
\label{theorem2}
Let $K\in \mathcal{K}^{d},\,m\in\mathbb{N}$, and $j\in\mathbb{N}$. Then the probability that $j$ points distributed independently and uniformly in $K$ are vertices of the convex hull of these $j$ and $m$ further points distributed independently and uniformly in $K$ is given by
\begin{equation*} 
\frac{E(N_{m+j})_{(j)}}{(m+j)_{(j)}},
\end{equation*}
where $E(N_{m+j})_{(j)}$ denotes the $j$-th factorial moment of $N_{m+j}$  and $(m+j)_{(j)}=(m+j)(m+j-1)\cdots(m+1)$.
\end{theorem}
\begin{proof} 
 On the one hand, for any $j$ out of $m+j$ points distributed independently and uniformly in $K$, the probability of being vertices of the convex hull of the $m+j$ points is the same. Hence, as there are ${m+j\choose j}$ possibilities to choose $j$ points out of $m+j$ points, ${m+j\choose j}$  times this probability gives the expected number of possibilities to choose $j$ points out of the $m+j$ points such that the chosen $j$ points are vertices of the convex hull of the $m+j$ points.
 
 On the other hand, the number of possibilities to choose $j$ points out of $m+j$ points such that the chosen $j$ points are vertices of the convex hull of the $m+j$ points is just the number of possibilities to choose $j$ points out of the vertices of the convex hull, i.e.\ ${N_{m+j}\choose j}$. Thus the expected number of possibilities is  $E{N_{m+j}\choose j}$.
 
Combining these two observations, we find that 
\begin{equation*} 
 \frac{E{N_{m+j}\choose j}}{{{m+j}\choose j}}=\frac{E(N_{m+j})_{(j)}}{(m+j)_{(j)}}
\end{equation*}
is just the probability in question.
\end{proof}

\section{Duality}
\label{sec:2}
\begin{theorem}[dual version of Theorem \ref{theorem1}]
\label{theorem3}
Let $K\in\mathcal{K}^{d}$, $n\in\mathbb{N}$, and $k\in\mathbb{N}$. Then
\begin{equation*} 
 \frac{E(N_{n+k})_{(k)}}{(n+k)_{(k)}}=\sum\limits_{j=0}^{k}(-1)^{j}{k\choose j}\frac{EV^{j}_{n+k-j}}{(\textnormal{vol}\;K)^{j}},
\end{equation*}
where $E(N_{n+k})_{(k)}=EN_{n+k}(N_{n+k}-1)\cdots(N_{n+k}-k+1)$ is the $k$-th factorial moment of $N_{n+k}$ and $(n+k)_{(k)}=(n+k)(n+k-1)\cdots(n+1)$.
\end{theorem}
\begin{proof} 
Theorem \ref{theorem1}, with $n$ and $k$ replaced by $n+k-i$ and $i\;(i=0,\ldots,k)$, can be written in the form $\textbf{v}^{(k)}_{n}=A_{k}\textbf{n}^{(k)}_{n}$, where
\begin{equation*} 
\textbf{v}^{(k)}_{n}=\left( 1,\frac{EV_{n+k-1}}{\textrm{vol}\;K},\frac{EV_{n+k-2}^{2}}{(\textrm{vol}\;K)^{2}},\ldots,\frac{EV^{k}_{n}}{(\textrm{vol}\;K)^{k}}\right)^{T}
\end{equation*}
and
\begin{equation*} 
\textbf{n}^{(k)}_{n}=\left( 1,\frac{EN_{n+k}}{n+k},\frac{E(N_{n+k})_{(2)}}{(n+k)_{(2)}},\ldots,\frac{E(N_{n+k})_{(k)}}{(n+k)_{(k)}}\right)^{T}
\end{equation*}
are $(k+1)$-vectors and 
\begin{equation*} 
A_{k}=\left( (-1)^{j}{i\choose j}\right) _{i=0,\ldots,k;\;j=0,\ldots,k}
\end{equation*}
is a $(k+1)\times(k+1)$-matrix.

The inverse of $A_{k}$ is $A_{k}$ itself. To see this, consider
\begin{equation*} 
A_{k}^{2}=\left( (-1)^{j}{i\choose j}\right) _{i=0,\ldots,k;\;j=0,\ldots,k}\;\left( (-1)^{l}{j\choose l}\right) _{j=0,\ldots,k;\;l=0,\ldots,k}.
\end{equation*}
The entries in the $i$-th row of the first matrix are zero if $j>i$, and the entries in the $l$-th column of the second matrix are zero if $j<l$. Consequently, if $i<l$, the inner product of the $i$-th row of the first matrix with the $l$-th column of the second matrix is zero. If $i\geq l$, the inner product of the $i$-th row of the first matrix with the $l$-th column of the second matrix is given by
\begin{equation*}
\sum\limits_{j=l}^{i}(-1)^{j+l}{i\choose j}{j\choose l}={i\choose l}\sum\limits_{j=l}^{i}(-1)^{j+l}{i-l\choose j-l}={i\choose l} \sum\limits_{j=0}^{i-l}(-1)^{j}{i-l\choose j},
       \end{equation*}
which is one if $i=l$ and zero if $i>l$.

Hence $\textbf{v}^{(k)}_{n}=A_{k}\textbf{n}^{(k)}_{n}$ implies $\textbf{n}^{(k)}_{n}=A_{k}\textbf{v}^{(k)}_{n}\!$, proving Theorem \ref{theorem3}. 
\end{proof}

Finally we point out that, likewise,  $\textbf{n}^{(k)}_{n}=A_{k}\textbf{v}^{(k)}_{n}$ implies $\textbf{v}^{(k)}_{n}=A_{k}\textbf{n}^{(k)}_{n}\!$, proving the duality of Theorems \ref{theorem1} and \ref{theorem3}.

\section{Concluding Remark}
\label{sec:3}
Alternatively, one could start with Theorem \ref{theorem2}, the proof of which is based on the same approach as the proof of Theorem 1 in \cite{RefBu}. As described in the Introduction, the inclusion--exclusion principle then implies the right-hand side of (\ref{eq:3}), and the Proposition verifies the equivalence of the right-hand side of (\ref{eq:2}) and the right-hand side of (\ref{eq:3}).

\vspace{1cm}

\footnotesize
\noindent
Department of Mathematics, Salzburg University, Austria \\
\texttt{christian.buchta@plus.ac.at}

\end{document}